\documentclass[12pt]{article}

\usepackage{amsthm, amsmath, amssymb}
\usepackage{enumerate}
\usepackage{thmtools, thm-restate}
\usepackage{microtype}
\usepackage{color}
\definecolor{darkgreen}{rgb}{0,0.35,0}
\definecolor{darkblue}{rgb}{0,0,0.6}
\usepackage[breaklinks,pdftex,colorlinks,citecolor=darkgreen,linkcolor=darkblue]{hyperref}

\newtheorem*{tutteslinkingtheorem}{Tutte's Linking Theorem}

\declaretheorem[numberwithin=section]{theorem}
\declaretheorem[numberlike=theorem]{lemma}
\declaretheorem[numberlike=theorem]{corollary}
\declaretheorem[numberlike=theorem]{conjecture}
\declaretheorem[numberlike=theorem]{proposition}
\newtheoremstyle{parentheses}{\topsep}{\topsep}{\itshape}{}{}{}{ }{\thmnumber{(#2)}}
\theoremstyle{parentheses}
\newtheorem{claim}{}
\makeatletter\@addtoreset{claim}{theorem}\makeatother

\renewcommand{\d}{\backslash}

\DeclareMathOperator{\si}{si}
\DeclareMathOperator{\cl}{cl}

\DeclareMathOperator{\GF}{GF}
\DeclareMathOperator{\F}{\mathbb{F}}
\newcommand{\defn}[1]{\textbf{#1}}

\title{Matroids with a modular 4-point line}
\author{Rohan Kapadia\thanks{This research was partially conducted at the University of Waterloo, Waterloo, Ontario, Canada. Current address: Concordia University, Montr\'eal, Qu\'ebec, Canada.}}

\date{September 20, 2013; revised March 18, 2014}

\begin{document}

\maketitle

\begin{abstract}
 A result of Seymour implies that any $3$-connected matroid with a modular $3$-point line is binary. We prove a similar characterization for $3$-connected matroids with modular $4$-point lines.
 We show that such a matroid is either representable over $\GF(3)$ or $\GF(4)$ or has an $F_7$-minor and either an $F_7^-$- or $(F_7^-)^*$-minor.
\end{abstract}

\section{Introduction}

We call a subset $X$ of the ground set of a matroid $M$ \defn{modular} if, for every flat $F$ of $M$,
\[ r_M(X) + r_M(F) = r_M(X \cup F) + r_M(X \cap F).\]

In many cases, the modularity of $X$ in $M$ forces certain structural properties of the restriction $M | X$ to be shared by $M$.
For instance, a theorem of Seymour implies the following characterization of matroids with modular $3$-point lines.

\begin{theorem}[Seymour, \cite{Seymour:u24roundedness}] \label{thm:seymourmodulartriangle}
 Every $3$-connected matroid with a modular $3$-point line is binary.
\end{theorem}

We call a line, or rank-two flat, a \defn{$k$-point line} if it is the union of $k$ points, or rank-one flats (in a simple matroid this is a line with $k$ elements).
Another example of modularity forcing structure on a matroid is this fact that also follows from a theorem of Seymour \cite{Seymour:triplesinmatroidcircuits}: if $M$ is a vertically $4$-connected matroid with a modular set $X$ such that $M|X \cong M(K_4)$, then $M$ is graphic. Modularity also plays a role in Geelen, Gerards, and Whittle's work on structure theorems for minor-closed classes of matroids representable over a finite field \cite{GeelenGerardsWhittle}.

\autoref{thm:seymourmodulartriangle} can be extended from the binary case to other finite fields if we consider modular rank-$3$ sets rather than lines: the main result of \cite{modularplanepaper} is that if $M$ is a vertically $4$-connected matroid with a modular rank-$3$ set $X$, then $M$ is representable over the same set of finite fields as $M | X$.
On the other hand, it is much more complicated to characterize the structure that is forced on matroids by modular lines of more than three points.
In this paper, we consider this problem for the case of modular $4$-point lines.
Our main theorem is the following.

\begin{restatable}{theorem}{fourpointlinetheorem} \label{thm:4ptlinetheorem}
If $M$ is a $3$-connected matroid with a modular $4$-point line then either
\begin{enumerate}[(i)]
 \item $M$ is ternary, \label{out:4ptlinetheorem-first}
 \item $M$ is quaternary, or \label{out:4ptlinetheorem-second}
 \item $M$ has an $F_7$-minor and either an $F_7^-$- or $(F_7^-)^*$-minor. \label{out:4ptlinetheorem-last}
\end{enumerate}
\end{restatable}

We denote by $F_7$ the binary projective plane, also called the Fano matroid, and by $F_7^-$ the non-Fano matroid, which is obtained from $F_7$ by relaxing a circuit-hyperplane.
Since $F_7$ is representable only over fields of characteristic two and $F_7^-$ is representable only over fields of characteristic different from two, we have this corollary:

\begin{corollary}
 Any $3$-connected, representable matroid with a modular $4$-point line is representable over $\GF(3)$ or $\GF(4)$.
\end{corollary}

Any $4$-point line is modular in a ternary matroid, so it is clear that being ternary should be one of the outcomes of \autoref{thm:4ptlinetheorem}.
When a $3$-connected matroid $M$ has a modular $4$-point line, $L$, one case in which $M$ is quaternary is when there is an element $e \in L$ such that $L \setminus \{e\}$ is modular in $M \d e$. In this case, $L \setminus \{e\}$ is a $3$-point line of $M \d e$ (because $M$ is simple), and $M \d e$ is also $3$-connected, so $M \d e$ is binary by \autoref{thm:seymourmodulartriangle}.

A $4$-point line need not be modular in an arbitrary quaternary matroid. We have seen two cases in which it is: when the matroid is also ternary, and when deleting one element yields a binary matroid. It seems likely that if the matroid is vertically $4$-connected, then these are the only two possibilities:

\begin{conjecture} \label{conj:vertically4connected}
If $M$ is a simple, vertically $4$-connected matroid with a modular $4$-point line, $L$, then either
\begin{enumerate}[(i)]
 \item $M$ is ternary,
 \item for some $e \in L$, $M \d e$ is binary, or
 \item $M$ has an $F_7$-minor and either an $F_7^-$- or $(F_7^-)^*$-minor.
\end{enumerate} 
\end{conjecture}

This is not the case when $M$ is $3$-connected but not vertically $4$-connected. We can construct other examples that are representable over $\GF(4)$ by taking a matroid $N$ that has a $4$-point line and is representable over both $\GF(3)$ and $\GF(4)$, picking any triangle $T$ of $N$, and taking the generalized parallel connection across $T$ of $N$ and a copy of $F_7$ (see \cite[Section 11.4]{Oxley} for the definition of the generalized parallel connection, which was introduced by Brylawski \cite{Brylawski}). Note that in this construction, $T$ may or may not be contained in the modular line.

It is possible that a variant of \autoref{conj:vertically4connected} holds even for modular lines with $q+1$ points when $q \leq 5$. That is, perhaps any simple, vertically $4$-connected representable matroid $M$ with a modular $(q+1)$-point line $L$ is either representable over $\GF(q)$ or has an element $e \in L$ such that $L \setminus \{e\}$ is a modular line of $M \d e$.
However, this does not hold for modular lines of seven points; one of the exceptions comes from the class of Dowling geometries.
Given a finite multiplicative group $G$, the rank-$3$ Dowling geometry over $G$, denoted $Q_3(G)$, can be defined as follows.
Its ground set is $\{a_1, a_2, a_3\} \cup G_1 \cup G_2 \cup G_3$, where $G_1, G_2$, and $G_3$ are disjoint copies of $G$; let $f_i : G_i \rightarrow G$ be an isomorphism for each $i = 1, 2, 3$. The matroid is simple and has rank $3$, so we define it by its collection of lines, which are
\begin{enumerate}
 \item $G_i \cup (\{a_1, a_2, a_3\} \setminus \{a_i\})$, for each $i = 1, 2, 3$,
 \item $\{a_i, g\}$, for each $i = 1, 2, 3$ and each $g \in G_i$, and
 \item $\{g_1, g_2, g_3\}$ for each $g_1 \in G_1$, $g_2 \in G_2$, and $g_3 \in G_3$ such that $f_1(g_1)f_2(g_2)f_3(g_3) = 1$.
\end{enumerate}
The matroid $Q_3(G)$ has three modular $(|G|+2)$-point lines, those of the first type in the above list. Dowling \cite{Dowling} showed that $Q_3(G)$ is representable over a field $\F$ if and only if $G$ is isomorphic to a subgroup of the multiplicative group of $\F$. 
So when $|G| = 5$, the smallest field over which $Q_3(G)$ is representable is $\GF(11)$, as $11$ is the smallest prime power that is one more than a multiple of $5$.

To end this introduction, we point out that a matroid satisfying the third outcome of \autoref{thm:4ptlinetheorem} and of \autoref{conj:vertically4connected} can be constructed on only ten elements. We start with a rank-$4$ \emph{binary spike} $N$; see \cite{Oxley} for the definition and discussion of spikes.
Let $t$ be the tip of $N$ and $C$ one of the $4$-element circuits that meets every leg. We let $M$ be the spike obtained from $N$ by relaxing the circuit-hyperplane $C$. Then for any leg $\{x, y\}$ with $y \in C$, $M / x \d y \cong F_7$ while $M / y \d x \cong F_7^-$. Finally, choosing a leg $\{a, b\}$ of $M$, we can construct a single-element extension of $M$ by adding an element $e$ that is in the closures of $\{t, a, b\}$ and of $C \setminus \{a, b\}$. This matroid has $\{t, a, b, e\}$ as a modular line.

\section{Duals of some small matroids}

We follow the notation of Oxley \cite{Oxley}.
In this section we present two lemmas about some matroids $N$ showing that, under certain conditions, a matroid has an $N$-minor if and only if it has an $N^*$-minor.
A proof of the first one can be found in \cite[Proposition 12.2.15]{Oxley}.

\begin{lemma}[Oxley, \cite{Oxley:acharacterization}] \label{lem:U25U35}
If $M$ is a $3$-connected matroid of rank and corank at least three, then $M$ has a $U_{2,5}$-minor if and only if $M$ has a $U_{3,5}$-minor.
\end{lemma}

Next, we look at $F_7$- and $F_7^*$-minors. 
We need this corollary of the Splitter Theorem of Seymour (see \cite[Lemma 12.3.11]{Oxley}).

\begin{lemma} \label{lem:deletablepoint}
If $M$ is a $3$-connected matroid with a $3$-connected minor $N$, $|E(N)| \geq 4$, and $r(M) > r(N)$, then $M$ has an element $e$ such that $\si(M / e)$ is $3$-connected and has an $N$-minor.
\end{lemma}

Seymour \cite{Seymour:decomposition} proved that any $3$-connected binary matroid with an $F_7$-minor is either a copy of $F_7$ or has an $F_7^*$-minor (see also \cite[Proposition~12.2.3]{Oxley}). We generalize this fact from binary matroids to matroids with no $U_{2,5}$-minor. This lemma implies that $F_7$ is a splitter for the class of matroids with no $U_{2,5}$- or $F_7^*$-minor.

\begin{lemma} \label{lem:F7F7*}
If $M$ is a $3$-connected matroid with an $F_7$-minor and no $U_{2,5}$-minor, then either $M \cong F_7$ or $M$ has an $F_7^*$-minor.
\end{lemma}

\begin{proof}
We let $M$ be a $3$-connected matroid with an $F_7$-minor and no $U_{2,5}$-minor. \autoref{lem:U25U35} implies that $M$ also has no $U_{3,5}$-minor.

\begin{claim} \label{clm:F7F7*-ext}
 If $N$ is a simple rank-$3$ matroid with $e \in E(N)$ such that $N \d e = F_7$, then $N$ has a $U_{2,5}$-minor.
\end{claim}

Any two lines of $F_7$ intersect in a point. Thus since $e$ is not parallel to an element of $F_7$, $e$ lies in the closure in $N$ of at most one line of $F_7$. Therefore, $N / e$ has at most one parallel class of size greater than one. Hence it is a rank-$2$ matroid with at least $|E(F_7)| - 2 = 5$ points. This proves (\ref{clm:F7F7*-ext}).

\begin{claim} \label{clm:F7F7*-coext}
 $M$ has a minor $N$ with an element $e \in E(N)$ such that $N / e \cong F_7$, and $e$ is neither a coloop nor in a series pair of $N$.
\end{claim}

If $r(M) = 3$ then $M \cong F_7$ by (\ref{clm:F7F7*-ext}), so we may assume that $r(M) > 3$.
It follows from \autoref{lem:deletablepoint} that we can repeatedly contract elements of $M$ and simplify to get a $3$-connected rank-$4$ minor $N'$ and an element $e \in E(N')$ such that $N' / e$ has an $F_7$-restriction. However, by (\ref{clm:F7F7*-ext}) any simple rank-$3$ extension of $F_7$ has a $U_{2,5}$-minor; thus $\si(N' / e) \cong F_7$. We choose $N$ to be a restriction of $N'$ containing $e$ and exactly one element from each parallel class of $N' / e$. We claim that we can choose $N$ so that $e$ is neither a coloop nor in a series pair of $N$.

Suppose that $e$ is a coloop of $N$; then since $N'$ is $3$-connected, there is an element $f \in E(N') \setminus E(N)$. By our choice of $N$, $f$ is in a triangle with $e$ and some element $g \in E(N)$.
The set $E(N) \setminus \{e, g\}$ contains a basis of $N / e \cong F_7$ so it has rank three in $N'$. Since $\{e, f, g\}$ is a triangle and $e$ is a coloop of $N$, $f \not\in \cl_{N'}(E(N) \setminus \{e\})$.
Thus we could have chosen the restriction $N' | ((E(N) \setminus \{g\}) \cup \{f\})$ in place of $N$, as $e$ is not a coloop in it. It does, however, have $\{e, f\}$ as a series pair. So we can choose $N$ so that $e$ is not a coloop.

Similarly, suppose that for some $h \in E(N)$, $\{e, h\}$ is a series pair of $N$. Since $N / e$ is cosimple, this is the unique series pair. Since $N'$ is $3$-connected there are elements $g \in E(N) \setminus \{h\}$ and $f \in E(N') \setminus E(N)$ such that $\{e, f, g\}$ is a triangle of $N'$. 
Since $E(N) \setminus \{e, h, g\}$ contains a basis of $N / e$, it has rank three in $N'$. So its closure contains $g$, but not $f$.
Then we can choose $N' | ((E(N) \setminus \{g\}) \cup \{f\})$ in place of $N$; $\{e, h, f\}$ is a triad in it so $e$ is neither a coloop nor in a series pair. This proves (\ref{clm:F7F7*-coext}).
\\

We let $N$ be the minor of $M$ as in (\ref{clm:F7F7*-coext}). 

\begin{claim} \label{clm:F7F7*-onlylineislifted}
 There is a $4$-element circuit $C$ of $N / e$ that is also a circuit in $N$.
\end{claim}

Suppose not; then for every $4$-element circuit $C$ of $N / e \cong F_7$, $r_N(C) = 4$. Let $L_1$ and $L_2$ be two distinct lines of $N / e$. If $r_N(L_1) = r_N(L_2) = 2$, then $(L_1 \cup L_2) \setminus (L_1 \cap L_2)$ is a $4$-element circuit of $N$, a contradiction. Hence there is at most one line $L$ of $N / e$ that is also a circuit of $N$. 
Therefore, for some $f \in E(N)$, $N \d f$ has no triangles.
If $X$ is a $4$-element circuit of $N$ and $e \not\in X$, then by assumption $X$ is not a circuit of $N / e$; but then $e \in \cl_N(X)$ and $r_{N/e}(X) = 2$, a contradiction because lines of $F_7$ have only three elements.
Hence every $4$-element circuit of $N$ contains $e$.
Thus $N \d e, f$ has no circuits of size less than five, so $N \d e, f \cong U_{4, 6}$. Then $M$ has a $U_{3, 5}$-minor, a contradiction. This proves (\ref{clm:F7F7*-onlylineislifted}).
\\

We let $C$ be the circuit as in (\ref{clm:F7F7*-onlylineislifted}). We note that $E(N) \setminus C$ is a triangle of $N / e$; we denote its elements by $\{f, g, h\}$. 
Suppose that $f \in \cl_N(C)$. Since $e$ is not a coloop or in a series pair, $\{e, g, h\}$ is a triad of $N$. 
If $f \in \cl_N(\{g, h\})$, then $N \d f \cong F_7^*$, and we are done. 
On the other hand, if $f \not\in \cl_N(\{g, h\})$, then $\{e, f, h\}$ is the unique triangle of $N / g$ containing $h$; hence $N / g, h \d e \cong U_{2, 5}$, a contradiction.
Therefore, we may assume that $f \not\in \cl_N(C)$, and by symmetry also that $g, h \not\in \cl_N(C)$, so $\{e, f, g, h\}$ is a cocircuit of $N$.

Each element $x \in \{f, g, h\}$ is in exactly two lines of $(N / e) | (C \cup \{x\})$; we denote them by $L_1(x)$ and $L_2(x)$.
If $\{g, h\} \cup (L_1(f) \setminus \{f\})$ and $\{g, h\} \cup (L_2(f) \setminus \{f\})$ are both circuits of $N$, then $N | (C \cup \{e, g, h\}) \cong F_7^*$. So we may assume that $\{g, h\} \cup (L_i(f) \setminus \{f\})$ is independent in $N$ for some $i \in \{1, 2\}$.
Similarly, $\{f, g\} \cup (L_j(h) \setminus \{h\})$ is independent in $N$ for some $j \in \{1, 2\}$. There is a unique element $z \in L_i(f) \cap L_j(h)$. 
It is not in $\cl_N(\{e, f, g, h\})$ so $\{f, h\}$ is independent in $N / g, z$.
The matroid $N / g, z \d e$ is isomorphic to $U_{2, 5}$, contradicting the fact that $M$ has no $U_{2,5}$-minor.
\end{proof}

\section{Contracting a minor onto a modular line}

In this section, we look at $3$-connected matroids that have both a modular line and an $N$-minor, for any given $3$-connected matroid $N$. We show that there are finitely many minor-minimal such matroids. In fact, we show that each of them has rank at most $r(N) + 2$.

Let $W$ be a modular set in a matroid $M$.
The following two properties are straightforward and we will use them freely throughout the rest of the paper.
\begin{itemize}
 \item For any $e \in E(M) \setminus W$, $W$ is modular in $M \d e$, and if $e \not\in \cl_M(W)$ then $W$ is modular in $M / e$. Equivalently, $W$ is modular in any minor of $M$ that has $M | W$ as a restriction.
 \item If $e \in \cl_M(W) \setminus W$, then either $e$ is a loop or $e$ is parallel to an element of $W$.
\end{itemize}

In a simple matroid $M$, every modular set is a flat (if $X$ is a modular set that is not closed, then $X$ along with any point in $\cl_M(X) \setminus X$ violates the definition of modularity).
Modularity is often defined only for flats. The reason that we have defined it for arbitrary sets is so that the first property above holds --- a modular flat $W$ in a matroid $M$ remains a modular set in any minor with $M | W$ as a restriction, even though $W$ may not be a flat of such a minor.

We define the \defn{connectivity function}, $\lambda_M$, on subsets of the ground set of a matroid $M$ by
\[ \lambda_M(X) = r_M(X) + r_M(E(M) \setminus X) - r(M). \]
The \defn{local connectivity} of two sets $S, T \subseteq E(M)$ is defined to be
\[ \sqcap_M(S, T) = r_M(S) + r_M(T) - r_M(S \cup T). \]
Finally, when $S$ and $T$ are disjoint subsets of $E(M)$, we define
\[ \kappa_M(S, T) = \min\{\lambda_M(A) : S \subseteq A \subseteq E(M) \setminus T\}. \]

We make use of the following useful theorem of Tutte \cite{Tutte}, which can be thought of as a matroid generalization of Menger's Theorem.

\begin{tutteslinkingtheorem}
Let $M$ be a matroid. If $S, T \subseteq E(M)$ are disjoint, then for each $e \in E(M) \setminus (S \cup T)$, either
\begin{itemize}
 \item $\kappa_{M \d e}(S, T) = \kappa_M(S, T)$, or
 \item $\kappa_{M / e}(S, T) = \kappa_M(S, T)$.
\end{itemize}
\end{tutteslinkingtheorem}

The next technical lemma will be used extensively in later sections. When a $3$-connected matroid $M$ has a modular line $L$ and a minor $N$, our objective is to find a smallest possible minor of $M$ that contains the line $L$, has an $N$-minor, and is $3$-connected.

\begin{lemma} \label{lem:squishedminor}
 Let $M$ be a simple, $3$-connected matroid with a modular line $L$ such that $|L| \geq 4$. If $N_0$ is a simple, $3$-connected minor of $M$,
 then there is a minor $N'$ of $M$ such that
 \begin{enumerate}[(i)]   
  \item $N'$ has an $N_0$-minor, $N$, \label{out:first}
  \item $E(N') = L \cup E(N)$,
  \item $M | L$ is a restriction of $N'$, and \label{out:third}
  \item $N'$ is $3$-connected. \label{out:fourth}
 \end{enumerate}
\end{lemma}

\begin{proof}
 If $|E(N_0)| \leq 2$ then $N' = M|L$ works, so we may assume that $|E(N_0)| > 2$.
 For each $N_0$-minor, $N$, of $M$, there is a set $C \subset E(M)$ such that $M / C$ has $N$ as a restriction. We choose an $N_0$-minor $N$ of $M$ and corresponding set $C$ such that $|C|$ is minimum, and subject to this, such that $|(E(N) \cup C) \cap L|$ is maximum.
 Note that $C$ is independent, and setting $D = E(M) \setminus (E(N) \cup C)$, we have $N = M / C \d D$.
 
 If $|E(N) \cap L| \geq 2$, then since $N$ is simple, $C \cap L = \emptyset$ and furthermore $M|L$ is a restriction of $M / C$. We set $N' = M / C \d (D \setminus L)$. We note that any parallel pair in $N'$ contains one element each of $E(N)$ and $L$, so our maximum choice of $|E(N) \cap L|$ implies that $N'$ is simple. Moreover, $N'$ is $3$-connected because it has $N$ as a spanning restriction. Hence $N'$ is our desired minor.
 
 Thus we may assume that $|E(N) \cap L| \leq 1$. Since $C$ is independent, $|C \cap L| \leq 2$. Moreover, if $|C \cap L| = 2$ then the elements of $L \setminus C$ are loops of $M / C$ and are thus contained in $D$ because $N$ is simple. Therefore, $|(E(N) \cup C) \cap L| \leq 2$, and we may choose a set $T \subseteq L \setminus (E(N) \cup C)$ of size two.
 
 We let $S = E(N) \cup C$. The sets $S, T \subset E(M)$ are disjoint and each has size at least two. Thus as $M$ is $3$-connected, $\kappa_M(S, T) = \lambda_M(T) = 2$.
 Suppose that $e$ is an element of $E(M) \setminus (\cl_M(S) \cup \cl_M(T))$. By Tutte's Linking Theorem, at least one of $\kappa_{M \d e}(S, T)$ and $\kappa_{M / e}(S, T)$ is equal to two. Moreover, both $M \d e$ and $M / e$ have $M | S$ and $M | L$ as restrictions. Therefore, we can repeatedly remove elements that are not in the closure of $S$ or $T$ by either deletion or contraction until we are left with a minor $M'$ of $M$ such that
 \begin{itemize}
  \item $E(M') = \cl_{M'}(S) \cup \cl_{M'}(T)$,
  \item $M|S$ is a restriction of $M'$ hence $N$ is a minor of $M' / C$, 
  \item $M | L$ is a restriction of $M'$, and 
  \item $\kappa_{M'}(S, T) = 2$.
 \end{itemize}
 Suppose that $\sqcap_{M'}(S, T) \leq 1$. Then $\sqcap_{M'}(\cl_{M'}(S), \cl_{M'}(T)) \leq 1$ and so 
 \[(\cl_{M'}(S) \setminus T, \cl_{M'}(T) \setminus (\cl_{M'}(S) \setminus T))\]
 is a $2$-separation of $M'$, contradicting the fact that $\kappa_{M'}(S, T) = 2$. Therefore, $\sqcap_{M'}(S, T) = 2$ and we have $T \subseteq \cl_{M'}(S)$. Hence $L \subseteq \cl_{M'}(S)$.
 
 We let $C'$ be a maximal subset of $C$ such that $r_{M' / C'}(L) = 2$, and let $C'' = C \setminus C'$.
 We set $N' = (M' / C') | (L \cup E(N))$.
 Then $E(N') = L \cup E(N)$. Since $M | L$ is a restriction of $M'$ and $r_{M' / C'}(L) = 2$, $M | L$ is a restriction of $N'$. Since $N$ is a minor of $M' / C$, it is a minor of $N'$.
 Hence (\ref{out:first})-(\ref{out:third}) hold for $N'$. To show that $N'$ is $3$-connected and (\ref{out:fourth}) holds, we need the following claim.
 
 \begin{claim} \label{clm:squishedminor-claim}
  $L \subseteq \cl_{N'}(E(N))$, or equivalently, $\sqcap_{N'}(E(N), L) = 2$.
 \end{claim}
 
 Our choice of $C'$ implies that $C'' \subseteq \cl_{M' / C'}(L)$. By the modularity of $L$, each element of $C''$ is parallel to an element of $L$ in $M' / C'$. Moreover, as we chose $C$ with $|(E(N) \cup C) \cap L|$ maximum, we have $C'' \subset L$.
 
 Since $L \subseteq \cl_{M'}(S)$ and $C' \subseteq S$, we have $L \subseteq \cl_{N'}(E(N) \cup C'')$. 
 If $\sqcap_{N'}(E(N), L) = 0$, then this implies that $|C'' \cap L| = 2$ and that $N = (M' / C'') | E(N) = M' | E(N)$, contradicting the minimality of $|C|$.
 
 We may thus assume that $\sqcap_{N'}(E(N), L) = 1$. By the modularity of $L$, there is an element $e \in L \cap \cl_{N'}(E(N))$. Since $L \subseteq \cl_{N'}(E(N) \cup C'')$, there is an element $f$ of $C''$ that is contained in $L \setminus \{e\}$. If $C'' = \{f\}$, then as $N$ is simple, $N = (N' / f) | E(N) = N' | E(N)$, contradicting the minimality of $|C|$. Otherwise, $|C''| \geq 2$; but then again as $N$ is simple we have $(N' / C'') | E(N) = (N' / e) | E(N)$, contradicting the minimality of $|C|$.
 This proves (\ref{clm:squishedminor-claim}).
 \\
 
 Since $N$ and $M | L$ are both connected and $E(N') = E(N) \cup L$, the only possible $1$-separation of $N'$ (up to ordering) is $(E(N), L)$; but this cannot be a $1$-separation by (\ref{clm:squishedminor-claim}). So $N'$ is connected. Any parallel pair of $N'$ contains one element each of $E(N)$ and $L$, so $N'$ is simple by our maximal choice of $|(E(N) \cup C) \cap L|$.
 
 Suppose that $N'$ has a $2$-separation, $(A, B)$. The fact that $N'$ is simple means that $A \not\subseteq \cl_{N'}(B)$ and $B \not\subseteq \cl_{N'}(A)$. If $E(N) \subseteq \cl_{N'}(A)$ then $B \setminus \cl_{N'}(A) \subseteq L$, contradicting (\ref{clm:squishedminor-claim}). So $E(N)$ is not contained in $\cl_{N'}(A)$, and by symmetry is not contained in $\cl_{N'}(B)$ either. As $N$ is $3$-connected, one of $A$ and $B$ contains only one element of $E(N)$; we may assume by symmetry that $B$ does, and we call this element $e$. Note that $e \not\in \cl_{N'}(A)$. Since $B$ is not contained in $\cl_{N'}(A)$ and $N'$ is connected, $|B \setminus \cl_{N'}(A)| \geq 2$. Thus there is an element of $L$ in $B \setminus \cl_{N'}(A)$. Since $|L| \geq 3$ and $L$ is a line of $N'$, we have $L \subseteq \cl_{N'}(B)$.
 Since $L \subseteq \cl_{N'}(E(N))$, we have $r(N') = r_{N'}(E(N')) = r_{N'}(A) + 1$. So $\sqcap_{N'}(A, L) = 1$ and by the modularity of $L$, there is an element $f \in L \cap \cl_{N'}(A)$.
 Now we conclude that $C''$ consists of a single element $g \in L \setminus \{e, f\}$ so that $N = N' / g \d (L \setminus \{e, g\})$. But this is isomorphic to $N' | (A \cup \{f\})$ (under the bijection between $A \cup \{e\}$ and $A \cup \{f\}$ that fixes the elements of $A$ and sends $e$ to $f$). Therefore, $N'$ has an $N$-restriction, contradicting the minimality of $|C|$.
 This proves that $N'$ is $3$-connected.
\end{proof}

\section{Finding an $F_7$-minor}

In this section we apply \autoref{lem:squishedminor} to matroids with modular lines and $U_{2,5}$-minors.
This result will be used several times in the next section, where we analyze matroids that have $U_{2,6}$-, $U_{4,6}$-, and $P_6$-minors.
We say that a minor $N$ of a matroid $M$ \defn{uses} an element $e \in E(M)$ if $e \in E(N)$.

\begin{lemma} \label{lem:U25givesfano}
If $M$ is a $3$-connected matroid with a $U_{2,5}$-minor and a modular $4$-point line, $L$, then $M$ has an $F_7$-minor that uses three elements of $L$.
\end{lemma}

\begin{proof}
We let $M$ be a $3$-connected matroid with a modular $4$-point line, $L = \{u, v, w, x\}$, and a $U_{2,5}$-minor.
By \autoref{lem:squishedminor}, we may assume that $E(M) = L \cup E(N)$, where $N$ is a $U_{2,5}$-minor of $M$.
We note that $r(M) \leq r(N) + r_M(L) = 4$.

Let $E(N) = \{a, b, c, d, e\}$.
If $r(M) = 2$ then both $M | L$ and $N$ are restrictions of $M$. Then $E(N) \subseteq \cl_M(L)$ and the modularity of $L$ implies that each element of $E(N)$ is parallel to a distinct element of $L$, a contradiction because $|E(N)| > |L|$.

\begin{claim} \label{clm:U25givesfano-rank3}
 If $r(M) = 3$ then $M$ has an $F_7$-minor that uses three elements of $L$.
\end{claim}

Assume that $r(M) = 3$. Then there is an element $y \in L$ such that $M / y$ has $N$ as a restriction; we may assume that $y = x$. Then no triangle of $M$ contains $x$ and two elements of $E(N)$.
Also, at most one element of $E(N)$ is contained in $L$. We may assume that $a, b, c, d \not\in L$. 

We note that every element of $E(N) \setminus L$ is contained in $\cl_M(\{a, z\})$ for some $z \in L$.
Suppose that $\cl_M(\{a, u\}) \setminus \{a, u\}$ contains at least two elements of $E(N)$; we may assume they are $b$ and $c$. Since $(L \setminus \{u\}, \cl_M(\{a, u\}))$ is not a $2$-separation of $M$, there is an element of $E(N)$ that is neither in $L$ nor in $\cl_M(\{a, u\})$. So we may assume that $d \in \cl_M(\{a, v\})$. The modularity of $L$ implies that each of $a, b$, and $c$ is parallel to a distinct element of $L \setminus \{u\}$ in $M / d$. So one of $a, b$ or $c$ is parallel to $x$ in $M / d$, contradicting the fact that no triangle contains $x$ and two elements of $E(N)$. Hence $\cl_M(\{a, u\}) \setminus \{a, u\}$ contains at most one element of $E(N)$. By symmetry the same is true for $\cl_M(\{a, v\}) \setminus \{a, v\}$ and $\cl_M(\{a, w\}) \setminus \{a, w\}$. 
Thus we may assume that $b \in \cl_M(\{a, u\})$, $c \in \cl_M(\{a, v\})$, and $d \in \cl_M(\{a, w\})$.

The modularity of $L$ implies that $\{b, c\}$ is in a triangle with some element $z$ of $L$. Since $c \not\in \cl_M(\{a, u\})$, $z \neq u$, and similarly, $z \neq v$. Also, $z \neq x$ as $\{b, c\}$ is not a parallel pair of $M / x$. So $\{w, b, c\}$ is a triangle. By a similar argument, $\{u, c, d\}$ and $\{v, b, d\}$ are triangles. Hence $M | \{a, b, c, d, u, v, w\} \cong F_7$, which proves (\ref{clm:U25givesfano-rank3}).
\\

We may assume that $r(M) = 4$. This means $M / L = N$.

\begin{claim} \label{clm:U25givesfano-notriangle}
 $M$ has no triangles not contained in $L$.
\end{claim}

No triangle of $M$ contains exactly one element of $L$ since $M / L = N$ is simple. Hence if there is a triangle of $M$ not contained in $L$, it is disjoint from $L$. We may assume that $\{a, b, c\}$ is a triangle. No element of $L$ is in $\cl_M(\{a, b, c\})$, so the modularity of $L$ implies that $\sqcap_M(\{a, b, c\}, L) = 0$. We have $d, e \not\in \cl_M(\{a, b, c\})$ for otherwise $(E(N), L)$ would be a $2$-separation. Thus each of $\cl_M(\{a, b, c, d\})$ and $\cl_M(\{a, b, c, e\})$ is a rank-$3$ flat so the modularity of $L$ implies that each contains an element of $L$. Since $(E(N), L)$ is not a $2$-separation, they each contain a distinct element of $L$. Hence we may assume that $u \in \cl_M(\{a, b, c, d\})$ and $v \in \cl_M(\{a, b, c, e\})$.

This means that $r_M(E(N)) = r(M) = 4$ and so $\sqcap_M(\{a, b, c\}, \{d, e\}) = 0$. Therefore, $\{a, b, c\}$ contains no parallel pairs of $M / d,e$, so each element of $\{a, b, c\}$ is parallel to a distinct element of $L$ in $M / d, e$. So there is a parallel pair of $M / d, e$ containing one of $u$ or $v$ and one of $a, b$, or $c$. We may assume by symmetry that it is $\{u, a\}$; hence $\{u, a, d, e\}$ is a circuit of $M$.
But $\{u, d, a, b\}$ is also a circuit, which implies that $r_M(\{u, a, b, d, e\}) = 3$. This is a contradiction since $r_M(\{a, b, d, e\}) = r_M(E(N)) = 4$. This proves (\ref{clm:U25givesfano-notriangle}).
\\

It follows from (\ref{clm:U25givesfano-notriangle}) that every $3$-element subset $X$ of $E(N)$ is independent in $M$, so it satisfies $\sqcap_M(X, L) = 1$ and there is a $4$-element circuit consisting of $X$ and one element of $L$. We denote this element of $L$ by $\phi(X)$.

There are ten $3$-element subsets of $E(N)$, so we may assume that $u$ is equal to $\phi(X)$ for at least three such sets $X$. Then there are two $3$-element subsets $X_1, X_2$ of $E(N)$ such that $\phi(X_1) = \phi(X_2) = u$ and such that $|X_1 \cap X_2| = 2$. Hence $|X_1 \cup X_2| = 4$. The sets $X_1 \cup \{u\}$ and $X_2 \cup \{u\}$ are circuits, so $X_1 \cup X_2$ has rank three and it is also a circuit.

Suppose that for some $z \in L \setminus \{u\}$, there are at least three $3$-element subsets $X$ of $E(N)$ such that $\phi(X) = z$. Then by the same argument that we applied to $u$, there is a $4$-element circuit $Y \subset E(N)$ such that $r_M(Y \cup \{z\}) = 3$. But $|Y \cap (X_1 \cup X_2)| = 3$ and both $Y$ and $X_1 \cup X_2$ are circuits, implying that $r_M(E(N)) = 3$. But that means that $(E(N), L)$ is a $2$-separation, a contradiction.
Therefore, $u$ is the unique element of $L$ that is equal to $\phi(X)$ for at least three $3$-element subsets $X \subset E(N)$.

We may assume that $e$ is the unique element of $E(N) \setminus (X_1 \cup X_2)$. We note that $e \not\in \cl_M(X_1 \cup X_2)$ for then $(E(N), L)$ would be a $2$-separation of $M$.
This implies that $u \neq \phi(X)$ for any $3$-element set $X \subset E(N)$ containing $e$. Hence $u = \phi(X)$ for precisely four $3$-element subsets of $E(N)$. In particular, these are the four $3$-element subsets of $X_1 \cup X_2$.
Therefore, each of $v$, $w$, and $x$ is equal to $\phi(X)$ for precisely two $3$-element subsets $X \subset E(N)$, and these subsets all contain $e$.

Hence in the matroid $M / e \d u$, $\{v, w, x\}$ is a triangle, and each element $z$ of $\{v, w, x\}$ is in precisely two triangles $T_1$ and $T_2$ whose other elements are contained in $X_1 \cup X_2 = \{a, b, c, d\}$. In each case, $T_1 \cap T_2 = \{z\}$ otherwise $z$ would be in a four-point line and hence contained in more than two such triangles.
So each element of $\{a, b, c, d\}$ is also in three triangles of $M / e \d u$, one containing each of $u,v$, and $w$.
So $M / e \d x$ is a $7$-element, rank-$3$ matroid in which every element is in precisely three triangles. Therefore, it is isomorphic to $F_7$.
\end{proof}

Lemmas~\ref{lem:U25givesfano}, \ref{lem:U25U35}, and \ref{lem:F7F7*} together imply that any $3$-connected matroid with a modular $4$-point line that has an $F_7^*$-minor also has an $F_7$-minor. This is the reason that outcome (\ref{out:4ptlinetheorem-last}) of \autoref{thm:4ptlinetheorem} can guarantee the existence of an $F_7$-minor rather than only an $F_7$- or an $F_7^*$-minor.
These lemmas do not hold with $F_7^-$ in place of $F_7$, causing the lack of symmetry between these two matroids in \autoref{thm:4ptlinetheorem}.

\section{Excluded minors}

In this section, we mention two excluded-minor characterizations of representability over $\GF(3)$ and $\GF(4)$, and we prove that a $3$-connected matroid with a modular four-point line and no $F_7^-$- or $(F_7^-)^*$-minor is ternary or quaternary.
This allows us to easily finish the proof of our main theorem.

\begin{theorem}[Bixby and Reid, \cite{Bixby}; Seymour, \cite{Seymour:ternaryexcludedminors}] \label{thm:reidstheorem}
The excluded minors for the class of ternary matroids are $U_{2,5}$, $U_{3,5}$, $F_7$, and $F_7^*$.
\end{theorem}

The next theorem involves the matroids $P_6$, $P_8''$ and $S(5,6,12)$.
The matroid $P_6$ is the six-element simple, rank-$3$ matroid with exactly one triangle, whose complement is a triad.

The only properties of $S(5,6,12)$ and $P_8''$ that we need are that $S(5,6,12)$ is ternary and $P_8''$ has no four-point line; the precise definitions of these two matroids are as follows.
The matroid $S(5,6,12)$ is represented over $\GF(3)$ by the matrix
$$\left(\begin{array}{ccccc|c}
  & & I_6 & & & {\begin{array}{rrrrrr}
 0 & 1 & 1 & 1 & 1 & 1 \\
 1 & 0 & 1 & -1& -1& 1 \\
 1 & 1 & 0 & 1 & -1& -1 \\
 1 & -1& 1 & 0 & 1 & -1 \\
 1 & -1& -1& 1 & 0 & 1 \\
 1 & 1 & -1& -1& 1 & 0 
\end{array}}
\end{array}\right).$$
The properties of this matroid are discussed in \cite{Oxley}; in particular, it has a $5$-transitive automorphism group. The matroid $P_8$ is obtained from $S(5,6,12)$ by deleting two elements and contracting two elements, and the matroid $P_8''$ is obtained from $P_8$ by relaxing its unique pair of disjoint circuit-hyperplanes.

\begin{theorem}[Geelen, Oxley, Vertigan, Whittle, \cite{GeelenOxleyVertiganWhittle}] \label{thm:threeexcludedminors}
If $M$ is a $3$-connected non-$\GF(4)$-representable matroid, then either
\begin{enumerate}[(i)]
 \item $M$ has a $U_{2,6}$-, $U_{4,6}$-, $P_6$-, $F_7^-$-, or $(F_7^-)^*$-minor,
 \item $M$ is isomorphic to $P_8''$, or
 \item $M$ is isomorphic to a minor of $S(5,6,12)$ with rank and corank at least $4$.
\end{enumerate}
\end{theorem}

As an application of \autoref{thm:threeexcludedminors}, we prove that any $3$-connected matroid $M$ with a modular $4$-point line and no $F_7^-$ or $(F_7^-)^*$-minor is ternary or quaternary. To prove this using \autoref{thm:threeexcludedminors}, we need to show that $M$ has no $U_{2, 6}$-, $U_{4,6}$-, or $P_6$-minor. We do that in a sequence of four lemmas, after first proving the following small result.

\begin{lemma} \label{lem:fanopluslinehasnoextension}
 Let $M$ be a simple rank-$3$ matroid with a modular $4$-point line $L$ and $x \in L$. If $M \d x$ has an $F_7$-restriction $P$ that uses three elements of $L$, then $M \d x = P$.
\end{lemma}

\begin{proof}
 Suppose that $M \d x$ has an $F_7$-restriction $P$ that uses three elements of $L$ and that $M \d x \neq P$. Then there is an element $f \in E(M \d x) \setminus E(P)$. Since $L$ is modular and $M$ is simple, $f \not\in \cl_M(L)$. We note that $P$ is a projective plane so any two lines of $P$ have a non-empty intersection. Hence there is at most one line $F$ of $P$ with the property that $f \in \cl_M(F)$, for otherwise $f$ would be parallel to an element of $P$. Thus there are two elements $a, b \in E(P) \setminus L$ such that neither $a$ nor $b$ is parallel to any element of $E(P)$ in $M / f$. Then the modularity of $L$ implies that $\{x, f, a\}$ and $\{x, f, b\}$ are triangles. But then $\{f, a, b\}$ is also a triangle so $a$ and $b$ are parallel in $M / f$, a contradiction.
\end{proof}

We recall the circuit elimination axiom, which we will use several times in the rest of this section: if $C_1$ and $C_2$ are circuits of a matroid $M$ and $e \in C_1 \cap C_2$, then there is a circuit of $M$ contained in $(C_1 \cup C_2) \setminus \{e\}$.

\begin{lemma} \label{lem:nou26}
 If $M$ is a $3$-connected matroid with a modular $4$-point line then $M$ has no $U_{2,6}$-minor.
\end{lemma}

\begin{proof}
Let $M$ be a minor-minimal $3$-connected matroid with a modular $4$-point line, $L = \{u, v, w, x\}$, and a $U_{2,6}$-minor, $N$.
\autoref{lem:squishedminor} implies that $E(M) = L \cup E(N)$.

Note that since $M$ is $3$-connected, $r_M(E(N)) = r(M)$.
Applying \autoref{lem:U25givesfano}, we conclude that $M$ has an $F_7$-minor $P$ that uses three elements of $L$.
Thus $3 \leq r(M) \leq r(N) + r_M(L) \leq 4$.

\begin{claim} \label{clm:nou26-notrank3}
 $r(M) = 4$
\end{claim}

If not, then $r(M) = 3$. In this case, the $F_7$-minor $P$ is a restriction of $M$.
Also, for some $z \in L$, $M / z$ has $N$ as a restriction. Hence $L$ contains at most one element of $E(N)$ so at least five elements of $E(N)$ are disjoint from $L$.
But this contradicts \autoref{lem:fanopluslinehasnoextension} which says that $M$ has exactly four elements disjoint from $L$.
This proves (\ref{clm:nou26-notrank3}).
\\

Since, $r(M) = 4$, $N = M / L$. Also, there is an element $e \in E(N)$ such that $M / e$ has the $F_7$-minor $P$ as a restriction. We may assume that $x$ is the unique element of $L \setminus E(P)$. Since $|E(P) \setminus L| = 4$, there is an element of $E(N)$ that is not in $E(P) \cup \{e\}$. We denote it by $f$.

\begin{claim} \label{clm:nou26-notriangle}
 $\{e, f\}$ is not contained in a triangle of $M$.
\end{claim}

For every $3$-element subset $X \subset (E(P) \setminus L) \cup \{e\}$, the set $X \setminus \{e\}$ is independent in $P$ and so $X$ is independent in $M$. Thus $\sqcap_M(X, L) \geq 1$; but $X$ has rank two in $M / L = N$ so $\sqcap_M(X, L) = 1$. Therefore there is a unique element $\phi(X)$ of $L$ such that $\phi(X) \in \cl_M(X)$. As $M / L = N$ is simple, $X \cup \phi(X)$ is a circuit.

Let $\{a, b, c\}$ be a $3$-element subset of $E(P) \setminus L$.
Since $M / e \d f, x$ is a copy of $F_7$ using $\{u, v, w\} = L \setminus \{x\}$, each of the three $2$-element subsets of $\{a, b, c\}$ is contained in a triangle of $M / e$ along with a distinct element of $\{u, v, w\}$. Hence \[\{\phi(\{a, b, e\}), \phi(\{b, c, e\}), \phi(\{c, a, e\})\} = \{u, v, w\}.\]
So if $\phi(\{a, b, c\}) \in \{u, v, w\}$ then $\{a, b, c, e\}$ is a circuit, which is a contradiction since $\{a, b, c\}$ is independent in $M / e$. Therefore, $\phi(\{a, b, c\}) = x$ and so $\{a, b, c, x\}$ is a circuit of $M$. 
Since $\{a, b, c\}$ was an arbitrary $3$-element subset of $E(P) \setminus L$, the same is true for all such subsets, which implies that $r_M(\{x\} \cup (E(P) \setminus L)) = 3$.

Now suppose that $\{e, f\}$ is contained in a triangle of $M$. Since $M / L$ is simple, the third element of the triangle is in $E(N) \setminus L$; call it $a$. Then $M / e \d a, x$ is isomorphic to $M / e \d f, x$ under the map that swaps $a$ with $f$ and fixes all other elements. So by the argument of the last paragraph with $f$ in place of $a$, we conclude that $r_M(\{x, f\} \cup E(P) \setminus (L \cup \{a\})) = 3$.
The fact that both $\{x\} \cup (E(P) \setminus L)$ and $\{x, f\} \cup E(P) \setminus (L \cup \{a\})$ have rank three means that $\{f\} \cup (E(P) \setminus L)$ does too. Note that $\{f\} \cup (E(P) \setminus L) = E(N) \setminus \{e\}$. But $e \in \cl_M(\{a, f\})$ so we conclude that $r_M(E(N)) = 3$, a contradiction because $r_M(E(N)) = r(M)$. This proves (\ref{clm:nou26-notriangle}).
\\

Note that $M / e \d f$ has ground set $E(P) \cup L$ so it is simple. Since $\{e, f\}$ is not in a triangle of $M$, $M / e$ is also simple.
But this contradicts \autoref{lem:fanopluslinehasnoextension} which asserts that $E(M / e) = E(P) \cup L$.
\end{proof}

\begin{lemma} \label{lem:nou46}
 If $M$ is a $3$-connected matroid with a modular $4$-point line then $M$ has no $U_{4,6}$-minor.
\end{lemma}

\begin{proof}
Let $M$ be a minor-minimal $3$-connected matroid with a modular $4$-point line, $L = \{u, v, w, x\}$, and a $U_{4,6}$-minor, $N$.
\autoref{lem:squishedminor} implies that $E(M) = L \cup E(N)$.

Note that since $M$ is $3$-connected, $L$ is coindependent and $r_M(E(N)) = r(M)$.
Also, $M$ has corank at least four. Then since $U_{4,6}$ has a $U_{3,5}$-minor \autoref{lem:U25U35} implies that $M$ has a $U_{2,5}$-minor. Thus \autoref{lem:U25givesfano} implies that $M$ has an $F_7$-minor $P$ that uses three elements of $L$.
As $r(M) \leq r(N) + r_M(L)$, we have $4 \leq r(M) \leq 6$.

\begin{claim} \label{clm:nou46-rank6}
 $r(M) < 6$.
\end{claim}

Suppose that $r(M) = 6$. Then $N = M / L$ and $r_M(E(N)) = 6$. Every $4$-element subset $Y$ of $E(N)$ is independent in $N$ and hence $\sqcap_M(Y, L) = 0$. But $\sqcap_M(E(N), L) = 2$, so for every $5$-element subset $X$ of $E(N)$ we have $\sqcap_M(X, L) = 1$, and so there is a unique element of $L$ contained in $\cl_M(X)$.
Since there are six $5$-element subsets of $E(N)$, there is an element $z \in L$ and two distinct $5$-element subsets $X_1, X_2 \subset E(N)$ such that $z \in \cl_M(X_1)$ and $z \in \cl_M(X_2)$. Since $z$ is not in the closure of any proper subset of $X_1$ or $X_2$, the two sets $X_1 \cup \{z\}$ and $X_2 \cup \{z\}$ are circuits. Hence $X_1 \cup X_2 = E(N)$ contains a circuit, a contradiction. This proves (\ref{clm:nou46-rank6}).

\begin{claim} \label{clm:nou46-rank5}
 $r(M) = 4$.
\end{claim}

If not, then $r(M) = 5$. Hence there are two elements $a, b \in E(N)$ such that $M / a, b$ has $P$ as a restriction. Since $|E(P) \setminus L| = 4$, this means that $E(N)$ is disjoint from $L$.
We also know that there is an element $z \in L$ such that $M / z$ has $N$ as a restriction; we may assume that $M / x$ has $N$ as a restriction.

Every $4$-element subset $X$ of $E(N)$ is independent in $M / x$ so $x \not\in \cl_M(X)$. So the modularity of $L$ implies that for each such set $X$, $\cl_M(X)$ contains exactly one of $u, v,$ or $w$. Hence no triangle of $M / a, b$ contains $x$ and so  $u, v, w \in E(P)$.
We denote by $\{c, d, e, f\}$ the set $E(N) \setminus \{a, b\}$. We note that this set is a circuit of $P$. By symmetry, we may assume that the triangles of $P$ are
\[ \{c, d, u\}, \{c, e, v\}, \{c, f, w\}, \{d, e, w\}, \{d, f, v\}, \{f, e, u\}, \{u, v, w\} .\]

By symmetry we may assume that $u$ is the element of $L$ contained in $\cl_M(\{b, c, d, e\})$. Hence there is a circuit $C_1$ with $u \in C_1 \subseteq \{b, c, d, e, u\}$.
Also, as $\{c, d, u\}$ is a circuit of $M / a, b$, there is a circuit $C_2$ of $M$ with $u \in C_2 \subseteq \{a, b, c, d, u\}$.
This implies that $\{a, b, c, d, e\}$ contains a circuit, contradicting the fact that $\{c, d, e\}$ is independent in $M / a, b$. This proves (\ref{clm:nou46-rank5}).
\\

Since $r(M) = 4$, there is an element $a \in E(N)$ such that $M / a$ has $P$ as a restriction. Since $N / a \cong U_{3,5}$, $N / a$ has no triangles. So $E(P) \cap L$ is disjoint from $E(N)$. But \autoref{lem:fanopluslinehasnoextension} implies that $|E(M / a) \setminus L| = 4$ so $L$ contains exactly one element of $E(N)$, and this element is not in $E(P)$. We may assume that $x \in E(N) \cap L$ and $x \not\in E(P)$.

We denote the set $E(N) \setminus \{a, x\}$ by $\{b, c, d, e\}$. By symmetry, we may assume that the triangles of $P$ are
\[ \{b, c, u\}, \{b, d, v\}, \{b, e, w\}, \{c, d, w\}, \{d, e, u\}, \{e, c, v\}. \]

Since $\{b, c, u\}$ is a circuit of $M / a$, there is a circuit $C_1$ of $M$ with $u \in C_1 \subseteq \{a, b, c, u\}$.
Note that $\sqcap_M(\{b, c, d\}, L) \geq 1$, so some element of $L$ is in $\cl_M(\{b, c, d\})$. It is not $x$, as $x \in E(N)$ and $\{b, c, d, x\}$ is independent. So we may assume by symmetry that $u \in \cl_M(\{b, c, d\})$. Then there is a circuit $C_2$ of $M$ with $u \in C_2 \subseteq \{b, c, d, u\}$. Hence $(C_1 \cup C_2) \setminus \{u\} = \{a, b, c, d\}$ contains a circuit of $M$, a contradiction.
\end{proof}

Finally, we consider $P_6$-minors. We actually prove that a $3$-connected matroid with a modular $4$-point line and a $P_6$-minor has an $F_7^-$-minor. As the proof is much longer than those for $U_{2,6}$ and $U_{4,6}$, it is divided into two lemmas. First, we prove that a minor-minimal $3$-connected matroid with a modular $4$-point line and a $P_6$-minor has rank four. Then we show that it has an $F_7^-$-minor. 

\begin{lemma} \label{lem:p6rank4}
 If $M$ is a minor-minimal $3$-connected matroid with a modular $4$-point line and a $P_6$-minor, then $r(M) = 4$.
\end{lemma}

\begin{proof}
Let $M$ be a minor-minimal $3$-connected matroid with a modular $4$-point line, $L = \{u, v, w, x\}$, and a $P_6$-minor, $N$.
\autoref{lem:squishedminor} implies that $E(M) = L \cup E(N)$. Thus $3 \leq r(M) \leq r(N) + r_M(L) = 5$.

Note that since $M$ is $3$-connected, $r_M(E(N)) = r(M)$.
Since $P_6$ has a $U_{2,5}$-minor, we can apply \autoref{lem:U25givesfano} and conclude that $M$ has an $F_7$-minor $P$ that uses three elements of $L$.

\begin{claim} \label{clm:nop6-rank3}
 $r(M) > 3$.
\end{claim}

Suppose that $r(M) = 3$. Then $M$ has both $N$ and $P$ as restrictions. We may assume that $x$ is the element of $L \setminus E(P)$. \autoref{lem:fanopluslinehasnoextension} implies that $E(M) = E(P) \cup \{x\}$. Thus some matroid obtained from $N \cong P_6$ by deleting at one most element is a restriction of $M \d x = P$.
Up to isomorphism, there are two matroids obtainable by deleting an element from $P_6$: $U_{3,5}$ and the $2$-sum $U_{2,4} \oplus_2 U_{2,3}$.
Both are non-binary so cannot be isomorphic to restrictions of $F_7$. This proves (\ref{clm:nop6-rank3}).
\\

We may assume that $r(M) = 5$, so we have $N = M / L \cong P_6$. Note that $E(N)$ is disjoint from $L$.
We recall that $M$ has an $F_7$-minor $P$ using three elements of $L$; hence $P$ is a restriction of $M / Z$ for some $2$-element subset $Z \subset E(N) \setminus L$.

We denote by $T^* = \{a, b, c\}$ the triad of $N$ and by $T = \{d, e, f\}$ the triangle of $N$.
We note that $T^*$ is also a triad of $M$. As $P$ is cosimple and has no triads, $M / Z$ has no triads either. This means that $T^* \cap Z \neq \emptyset$, so we may assume that $a \in Z$. We write $a'$ for the element of $Z \setminus \{a\}$.

\begin{claim} \label{clm:nop6-flat}
 $T^*$ is a flat of $M$.
\end{claim}

Since $T^*$ is independent in $M / L$, $\cl_M(T^*)$ contains no element of $L$; so if $T^*$ is not a flat then we may assume that $d \in \cl_M(T^*)$. Then $r_{M/a}(\{b, c, d\}) = 2$. 
Note that $M / a, a'$ is simple, because $L$ is a line of $M / a, a'$ and one element of $L$ is deleted from $M / a, a'$ to get $P$.
This implies that $a' \not\in \{b, c, d\}$. So $\{b, c, d\}$ is a triangle of $M / a,a'$ and hence of $P$. This is a contradiction because $P$ has no triangle disjoint from $L$, as $F_7$ has no pair of disjoint triangles. This proves (\ref{clm:nop6-flat}).
\\

It follows from (\ref*{clm:nop6-flat}) that for each $g \in T$, $r_M(T^* \cup \{g\}) = 4$. Hence $\sqcap_M(T^* \cup \{g\}, L) \geq 1$; but $T^*$ is independent in $M / L = N$ so we have $\sqcap_M(T^* \cup \{g\}, L) = 1$.
This means (by the modularity of $L$) that for each $g \in T$, there is a unique element $\phi(g)$ of $L$ such that $\phi(g) \in \cl_M(T^* \cup \{g\})$.

\begin{claim} \label{clm:nop6-twophis}
 $\phi(d)$, $\phi(e)$, and $\phi(f)$ are all distinct.
\end{claim}

We cannot have $\phi(d) = \phi(e) = \phi(f)$, otherwise $(E(N), L)$ would be a $2$-separation of $M$.
So we may assume that $\phi(d) = \phi(e) = x$ and $\phi(f) = w$. 
So $d, e \in \cl_{M / x}(T^*)$ and $r_M(T^* \cup \{d, e, x\}) = 4$. Since $M$ is $3$-connected, at least three elements of $\{f, u, v, w\}$ are disjoint from the hyperplane $H = \cl_M(T^* \cup \{d, e, x\})$. So $H$ does not contain any element of $L$ other than $x$. It also does not contain $f$ because $w \in \cl_M(T^* \cup \{f\})$. Thus $T^* \cup \{d, e, x\}$ is a flat of $M$.

We note that $\{d, e, x\}$ is independent in $M$ since $\{d, e\}$ is independent in $M / L$.
If for every $2$-element subset $Y_1 \subset T^*$ and every $2$-element subset $Y_2 \subset \{d, e, x\}$, $Y_1 \cup Y_2$ is independent, then $M | (T^* \cup \{d, e, x\}) \cong U_{4,6}$, contradicting \autoref{lem:nou46}. Hence we may assume that there are $2$-element sets $Y_1 \subset T^*$ and $Y_2 \subset \{d, e, x\}$ such that $Y_1 \cup Y_2$ is a circuit.

Note that no element of $T^*$ is in a triangle of $M / L = N$ so $x \not\in Y_2$, hence $Y_2 = \{d, e\}$.
Also, as $M / a,a'$ has no $2$- or $3$-element circuits disjoint from $L$, neither $a$ nor $a'$ is contained in $Y_1 \cup Y_2$. So $Y_1 = \{b, c\}$ and $a' = f$.

Since $w \in \cl_M(T^* \cup \{f\})$, $w \in \cl_{M / a, f}(\{b, c\})$ and $\{b, c, w\}$ is a triangle of $P$. 
No two triangles of $P$ are disjoint, so the triangle of $P$ containing $\{d, e\}$ contains an element of $\{b, c, w\}$ and of $L$, which means that $\{d, e, w\}$ is a triangle of $P$. So there is a circuit $C$ of $M$ such that $\{d, e, w\} \subseteq C \subseteq \{d, e, a, f, w\}$. As the intersection of a circuit and a cocircuit in a matroid is never equal to one, $|C \cap T^*| \neq 1$ so $a \not\in C$. Hence one of $\{d, e, w\}$ or $\{d, e, w, f\}$ is a circuit of $M$, so $r_{M / f}(\{d, e, w\}) \leq 2$; as $M / f, a$ is simple, this means $\{d, e, w\}$ is a triangle in $M / f$. 
Recall that $\{b, c, d, e\}$ is a circuit of $M$; as $f$ is not contained in the flat $T^* \cup \{d, e, x\}$ of $M$, $\{b, c, d, e\}$ is also a circuit of $M / f$. Then since $\{d, e, w\}$ is a triangle of $M / f$, one of $\{b, c, w\}$ or $\{b, c, e, w\}$ is a circuit of $M / f$. But if $\{b, c, w\}$ is a circuit of $M / f$, then $r_{M/w}(\{b, c, f\}) = 2$, a contradiction because no element of $T^*$ is contained in a triangle of the simple matroid $M / L = N \cong P_6$. So $\{b, c, e, w\}$ is a circuit of $M / f$.
Since $\{b, c, e, w\}$ has rank three in $P$ it has rank three in $M / a, f$, so $a \not\in \cl_{M / f}(\{b, c, e, w\})$. But then the fact that $\{b, c, e, w\}$ is a circuit of $M / f$ implies that it is a circuit of $M / a, f$. This is a contradiction because $\{b, c, w\}$ is a triangle of $P$ and hence of $M / a, f$. This proves (\ref{clm:nop6-twophis}).
\\

We may therefore assume that $\phi(d) = u$, $\phi(e) = v$ and $\phi(f) = w$.

\begin{claim} \label{clm:nop6-contracttwooftriad}
 $a' \in T$.
\end{claim}

If not, then $a' \in T^*$ and we may assume that $a' = b$ so $M / a, b$ has $P$ as a restriction. Moreover, $\{c, d, u\}$, $\{c, e, v\}$ and $\{c, f, w\}$ are triangles of $M / a, b$ so $P = M / a, b \d x$.
This means that the remaining triangles of $P$ are $\{d, e, w\}$, $\{e, f, u\}$, $\{f, d, v\}$, and $\{u, v, w\}$.

Since $T = \{d, e, f\}$ is not a triangle of $P$, it is also independent in $M$. The hyperplane $L \cup T$ has rank four, so $\sqcap_M(T, L) = 1$. So there is a unique element $z \in L$ such that $z \in \cl_M(T)$. Since $T$ contains no parallel pairs in $M / L$, the set $\{d, e, f, z\}$ contains no triangles of $M$ and is thus a circuit.

As $T = \{d, e, f\}$ is independent in $P$, we have $\sqcap_M(\{d, e, f\}, \{a, b\}) = 0$, which means that $\{d, e, f, z\}$ is also a circuit of $M / a,b$. This implies that $z \not\in \{u, v, w\}$, as $\{d, e, w\}$, $\{e, f, u\}$, and $\{f, d, v\}$ are triangles of $M / a, b$. So $z = x$, meaning that $\{d, e, f, x\}$ is a circuit of $M$.

The argument of the following paragraph holds also with $a$ replaced by $b$.
The set $\{a, c, d, e\}$ has rank four in $M$ because $\{c, d, e\}$ has rank three in $P$ and $M / a$. Hence $\sqcap_M(\{a, c, d, e\}, L) \geq 1$. The flat $\cl_M(\{a, c, d, e\})$ is a hyperplane of $M$, which is $3$-connected, so at least three elements of $M$ are disjoint from it; hence it contains at most one element of $L$. So $\sqcap_M(\{a, c, d, e\}) = 1$ and there is a unique element $y$ of $L$ in $\cl_M(\{a, c, d, e\})$. Now there is a circuit $C$ of $M$ with $y \in C \subseteq \{y, a, c, d, e\}$. Since $M / L$ is simple, $C \neq \{y, d, e\}$, so $C$ contains an element of $\{a, c\}$. But $T^*$ is a cocircuit and cannot intersect any circuit in exactly one element; hence $a, c \in C$. As $\{a, c\}$ is neither a parallel pair nor contained in a triangle of $M / L = N$, we conclude that $C = \{y, a, c, d, e\}$.
So $\{y, c, d, e\}$ is a circuit of $M / a$. As $\{c, d, e\}$ has rank three in $P$ (and in $M / a,b$), $b \not\in \cl_{M / a}(\{y, c, d, e\})$ and $\{y, c, d, e\}$ is a circuit of $M / a, b$. 
This means that $y \not\in \{u, v, w\}$, because $\{c, d, u\}$, $\{c, e, v\}$ and $\{d, e, w\}$ are triangles of $M / a,b$. Therefore, $y = x$ and $C = \{x, a, c, d, e\}$ is a circuit of $M$.
We note that by symmetry between $a$ and $b$, this also proves that $\{x, b, c, d, e\}$ is a circuit of $M$.

The fact that $\{d, e, f, x\}$ and $\{x, a, c, d, e\}$ are circuits means that there is a circuit $Y$ of $M$ such that $Y \subseteq \{a, c, d, e, f\}$. Moreover, $Y$ contains $f$ and at least one of $a$ or $c$; but $\{a, b, c\}$ is a cocircuit of $M$ so $a, c \in Y$.
As $M / a$ is simple, $Y \neq \{a, c, f\}$, and as neither $\{c, f, d\}$ nor $\{c, f, e\}$ is a triangle of $P$ or of $M / a$, $Y$ is not equal to $\{a, c, f, d\}$ or $\{a, c, f, e\}$. So $Y = \{a, c, d, e, f\}$ is a circuit. By the same argument with $a$ replaced by $b$, the fact that $\{x, b, c, d, e\}$ is a circuit means $\{b, c, d, e, f\}$ is also a circuit. 

Therefore, $r_M(\{a, b, c, d, e, f\}) = 4$, contradicting the fact that $r_M(E(N)) = r(M) = 5$. This proves (\ref{clm:nop6-contracttwooftriad}).
\\

We may assume that $a' = f$, so $M / a, f$ has $P$ as a restriction.
Then $\{b, c, w\}$ is a triangle of $M / a, f$ and of $P$. The triangle of $P$ containing $\{d, e\}$ meets both $\{b, c, w\}$ and $L$, so it is $\{d, e, w\}$.

The argument of this paragraph holds also with $c$ replaced by $b$.
We note that $r_M(\{a, c, d, e\}) = 4$ because $\{c, d, e\}$ is independent in $P$ and thus in $M / a$. So $\sqcap_M(\{a, c, d, e\}, L) \geq 1$. But $\{a, c, d, e\}$ has rank three in $N = M / L$ so we have $\sqcap_M(\{a, c, d, e\}) = 1$. Hence there is a unique element $z$ of $L$ such that $z \in \cl_M(\{a, c, d, e\})$, and there is a circuit $C$ such that $z \in C \subseteq \{z, a, c, d, e\}$. 
Since $\{a,c,d,e\}$ is a circuit of $N = M / L$, this circuit $C$ of $M$ must equal $\{z, a, c, d, e\}$.
By symmetry between $c$ and $b$, we also conclude that for some $z' \in L$, $\{z', a, b, d, e\}$ is a circuit.

Since $\{z, c, d, e\}$ is a circuit of $M / a$ and $\{c, d, e\}$ is independent in $M / a, f$, the set $\{z, c, d, e\}$ is also a circuit of $M / a, f$. If $z \in E(P)$ then $E(P) \setminus \{z, c, d, e\}$ consists of $b$ and two elements of $L$, a contradiction because each $4$-element circuit of $F_7$ is the complement of a triangle. So $z \not\in E(P)$.
The same argument applies to the circuit $\{z', b, d, e\}$ of $M / a$ and we conclude that $z' \not\in E(P)$. Therefore, as $L \setminus E(P)$ has only one element, $z = z'$ and both $\{z, a, c, d, e\}$ and $\{z, a, b, d, e\}$ are circuits of $M$.

Therefore, $\{a, b, c, d, e\} = T^* \cup \{d, e\}$ contains a circuit of $M$. So $r_M(T^* \cup \{d, e\}) \leq 4$. But $\cl_M(T^* \cup \{d, e\})$ contains $\{u, v\}$ and thus all of $L$. Since $f$ is not a coloop, it follows that $\cl_M(T^* \cup \{d, e\}) = E(M)$, contradicting the fact that $r(M) = 5$.
\end{proof}

Recall that two sets $A$ and $B$ in a matroid $M$ are called \defn{skew} if $\sqcap_M(A, B) = 0$.

\begin{lemma} \label{lem:nop6}
 If $M$ is a $3$-connected matroid with a modular $4$-point line and a $P_6$-minor, then $M$ has an $F_7^-$-minor.
\end{lemma}

\begin{proof}
We let $M$ be a minor-minimal $3$-connected matroid with a modular $4$-point line, $L = \{u, v, w, x\}$, and a $P_6$-minor, $N$. \autoref{lem:p6rank4} asserts that $r(M) = 4$. By \autoref{lem:squishedminor}, we may assume that $E(M) = L \cup E(N)$. Also, \autoref{lem:U25givesfano} implies that $M$ has an $F_7$-minor, $P$, using three elements of $L$.

We may assume that $M / x$ has $N$ as a restriction. Also, there is an element $e \in E(N)$ such that $M / e$ has $P$ as a restriction.

If $x \in E(P)$, then there are two triangles $T_1$ and $T_2$ of $P$ that contain $x$ and no other element of $L$. Then $T_1 \cup \{e\}$ and $T_2 \cup \{e\}$ both have rank three in $M$, so $(T_1 \cup \{e\}) \setminus \{x\}$ and $(T_2 \cup \{e\}) \setminus \{x\}$ both have rank two in $M / x$ and hence in $N$, a contradiction because $N \cong P_6$ is simple and has only one triangle. Therefore, $x \not\in E(P)$, and $E(P) \cap L = \{u, v, w\}$.

We let $Y = E(P) \setminus L$ and denote its elements by $\{a, b, c, d\}$. We let $f$ denote the element of $E(N) \setminus \{a, b, c, d, e\}$.
We let $\mathcal{Z}$ be the set of $2$-element subsets of $Y$.
For each $Z \in \mathcal{Z}$, either $Z$ and $L$ are skew in $M$ or $\sqcap_M(Z, L) = 1$.

\begin{claim} \label{clm:nop6-twopairs}
 There are distinct sets $Z, Z' \in \mathcal{Z}$ such that $\sqcap_M(Z, L) = \sqcap_M(Z', L) = 1$.
\end{claim}

We suppose that there is at most one $Z \in \mathcal{Z}$ such that $\sqcap_M(Z, L) = 1$. If such a $Z$ exists, we may assume that it equals $\{a, b\}$.
We set $X_1 = \{a, c, d\}$ and $X_2 = \{b, c, d\}$. Then every $2$-element subset of $X_1$ or of $X_2$ is skew to $L$. But $X_1$ and $X_2$ both have rank three, so $\sqcap_M(X_1, L) = \sqcap_M(X_2, L) = 1$.
Then there are elements $y_1, y_2 \in L$ such that $X_1 \cup \{y_1\}$ and $X_2 \cup \{y_2\}$ are circuits of $M$.
Since $X_1$ and $X_2$ also have rank three in $M / e$, $e \not\in \cl_M(X_1)$ and $e \not\in \cl_M(X_2)$. Hence $X_1 \cup \{y_1\}$ and $X_2 \cup \{y_2\}$ are circuits of $M / e$. As the only circuit of $P$ containing $X_1$ or $X_2$ is $Y$, $y_1$ and $y_2$ are not in $P$ so $y_1 = y_2 = x$.
Therefore, $X_1 \cup \{x\}$ and $X_2 \cup \{x\}$ are circuits of $M$. Then $X_1$ and $X_2$ are triangles of $M / x$ and hence of $N$, contradicting the fact that $P_6$ has only one triangle.
This proves (\ref{clm:nop6-twopairs}).

\begin{claim} \label{clm:nop6-twotriangles}
 There is a unique pair of sets $Z, Z' \in \mathcal{Z}$ such that $\sqcap_M(Z, L) = \sqcap_M(Z', L) = 1$. Moreover, $Z$ and $Z'$ are disjoint.
\end{claim}

By (\ref{clm:nop6-twopairs}) there are distinct sets $Z, Z' \in \mathcal{Z}$ such that $\sqcap_M(Z,L) = \sqcap_M(Z',L) = 1$. 
Since $|Y| = 4$, if (\ref{clm:nop6-twotriangles}) does not hold then we can choose $Z$ and $Z'$ so that $Z \cap Z' \neq \emptyset$.
We may assume that $Z = \{a, b\}$ and $Z' = \{a, c\}$. Now the set $L \cup \{a, b, c\}$ has rank three in $M$, so its closure is a hyperplane. As $M$ is $3$-connected, every cocircuit has size at least three. Since $|L \cup \{a, b, c\}| = 7$ and $|E(M)| \leq |L| + |E(N)| = 10$, $L \cup \{a, b, c\}$ is in fact a hyperplane and its complement, $\{d, e, f\}$ is a triad.
Note that $M / e \d f$ has $P$ as a restriction, so \autoref{lem:fanopluslinehasnoextension} implies that $f$ is in a parallel pair of $M / e$. As $\{d, e\}$ is a series pair of $M \d f$, $M / d \d f$ is isomorphic to $M / e \d f$, and it also has an $F_7$-restriction. So \autoref{lem:fanopluslinehasnoextension} implies that $f$ is in a parallel pair of $M / d$. Hence $\{e, f\}$ and $\{d, f\}$ are both contained in triangles of $M$; call them $T_e$ and $T_d$.
Since $r_M(L \cup \{a, b, c\}) = 3$, $\{a, b, c\}$ is a triangle of $M / x$ and hence of $N$. Since $P_6$ has only one triangle, this is the unique triangle of $N$, so neither $T_e$ nor $T_d$ is contained in $E(N)$. This means that $\sqcap_M(\{e, f\}, L) = \sqcap_M(\{d, f\}, L) = 1$. 
But then $r_M(L \cup \{d, e, f\}) = 3$ so $r_{M / x}(\{d, e, f\}) = 2$, contradicting the fact that $\{a, b, c\}$ is the unique triangle of $N$.
This proves (\ref{clm:nop6-twotriangles}).
\\

We let $Z, Z' \in \mathcal{Z}$ be as in (\ref{clm:nop6-twotriangles}); we may assume that $Z = \{a, b\}$ and $Z' = \{c, d\}$.
We may assume that $w \in \cl_M(\{a, b\})$. Since no two triangles of $P$ are disjoint, we also have $w \in \cl_M(\{c, d\})$, and $\{a, b, w\}$ and $\{c, d, w\}$ are triangles of $M$.
By symmetry between $u$ and $v$, we may assume that the other triangles of $P$ are $\{a, c, u\}$, $\{a, d, v\}$, $\{b, c, v\}$, $\{b, d, u\}$, and $\{u, v, w\}$.

We note that $r_M(Z \cup Z' \cup L) = 4$, otherwise $\cl_M(Z \cup Z' \cup L)$ would be a hyperplane of $M$ containing all but at most two elements, contradicting the fact that $M$ is $3$-connected. Also, $e$ is not contained in $\cl_M(Z \cup L)$, $\cl_M(Z' \cup L)$, or $\cl_M(Z \cup Z')$, since each of these sets has rank three in $M / e$.

If $f \in L$, then $\{a, b, f\}$ and $\{c, d, f\}$ would both be triangles of $M / x$, but $N$ has only one triangle; so $f \not\in L$. Since $M / e \d f$ has $P$ as a restriction, it follows from \autoref{lem:fanopluslinehasnoextension} that $f$ is contained in a parallel pair of $M / e$, so $\{e, f\}$ is contained in a triangle of $M$.

\begin{claim} \label{clm:nop6-nottriangleinN}
 Any triangle of $M$ containing $\{e, f\}$ is disjoint from $Y = E(N) \setminus \{e, f\}$.
\end{claim}

If not, then by symmetry we may assume that $\{e, f, a\}$ is a triangle in $M$.
Since $P_6$ has only one triangle, $\{e, f, a\}$ is the unique triangle of $M$ contained in $E(N)$, so $\{f, b, c\}$ is independent in $M$. Therefore, the modularity of $L$ means that for some $z \in L$, $z \in \cl_M(\{f, b, c\})$. 
We observe that $z \neq x$, for otherwise $\{f, b, c\}$ would be a triangle of $N$, whose unique triangle is $\{e, f, a\}$.
If $\{f, c, z\}$ were a triangle, then in $M / x$, $\{f, c, w\}$ would be a triangle, so we would have $r_{M/x}(\{f, c, d, w\}) = 2$, a contradiction because $N = M/x\d L$ has only one triangle, $\{e, f, a\}$. Therefore, $\{f, c, z\}$ is independent in $M$, and for the same reason, so is $\{f, b, z\}$. Moreover, no element of $L$ is in $\cl_M(\{b, c\})$ by (\ref{clm:nop6-twotriangles}). So $\{f, b, c, z\}$ is a circuit.
Also, $z \neq w$, otherwise we would have $f \in \cl_M(\{a, b, c, d\})$, which would imply $e \in \cl_M(\{a, b, c, d\})$, a contradiction.

Next, we suppose that $z = v$, so $\{f, b, c, v\}$ is a circuit of $M$. But since $\{b, c, v\}$ is a triangle of $M / e$, $\{b, c, v, e\}$ is also a circuit of $M$. This implies that $\{e, f, b, c\}$ contains a circuit; but $r_M(\{e, f, b, c\}) = 4$. So $z \neq v$ and we have $z = u$ and $\{f, b, c, u\}$ is a circuit.

Since $\{a, c, u\}$ is a triangle of $M / e$, the set $\{e, a, c, u\}$ is also a circuit of $M$. Hence $\{f, a, c, u\}$ is a dependent set. We consider all its $3$-element subsets.
The sets $\{a, c, u\}$ and $\{f, c, u\}$ are independent because they are proper subsets of the circuits $\{e, a, c, u\}$ and $\{f, b, c, u\}$, respectively.
Furthermore, $\{f, a, c\}$ and $\{f, a, u\}$ are independent because $e$ does not lie in $\cl_M(\{a, c\})$ or in $\cl_M(\{a\} \cup L)$.
Therefore, $\{f, a, c, u\}$ is a circuit.

Since $\{f, b, c, u\}$ and $\{f, a, c, u\}$ are both circuits, $r_M(\{f, a, b, c, u\}) = 3$. But $\cl_M(\{f, a, b, c, u\}) = E(M)$ and $M$ has rank four. This proves (\ref{clm:nop6-nottriangleinN}).

\begin{claim} \label{clm:nop6-fnotinplane}
 $f \not\in \cl_M(\{a, b, c, d\})$, $f \not\in \cl_M(L \cup \{a, b\})$, and $f \not\in \cl_M(L \cup \{c, d\})$.
\end{claim}

We know that $\{e, f\}$ is contained in a triangle of $M$ so (\ref{clm:nop6-nottriangleinN}) implies that $\{e, f, z\}$ is a triangle for some $z \in L$. 
Therefore, $f$ does not lie in $\cl_M(L \cup \{a, b\})$ or $\cl_M(L \cup \{c, d\})$ because $e$ does not.

We may assume that $f \in \cl_M(\{a, b, c, d\})$.
Clearly $z \neq x$ since $M / x$ has the simple matroid $N$ as a restriction.
If $z = w$, then $e \in \cl_M(\{f, w\}) \subset \cl_M(\{a, b, c, d\})$, a contradiction.
So we may assume that $z \in \{u, v\}$, and by symmetry between $u$ and $v$ we may assume that $z = u$ so $\{e, f, u\}$ is a triangle.

We recall that $r_{M / e}(\{a, c, u\}) = 2$, so $r_M(\{a, c, e, f, u\}) = 3$. Also, we have $r_M(\{a, b, c, d, f\}) = 3$. The intersection of two distinct planes in a matroid has rank at most two; so we have $r_M(\{a, c, f\}) = 2$.
Similarly, we have $r_M(\{b, d, f\}) = 2$ because $\{b, d, f\}$ is contained in the intersection of the distinct planes $\cl_M(\{b, d, e, f, u\})$ and $\cl_M(\{a, b, c, d, f\})$.
But then both $\{a, c, f\}$ and $\{b, d, f\}$ are triangles of $N$, a contradiction.
This proves (\ref{clm:nop6-fnotinplane}).
\\

The last claim implies that $\{a, b, c, d\}$ is a $4$-element circuit of $M / f$ and that none of $a, b, c$, or $d$ are in $\cl_{M / f}(L)$. Therefore, the modularity of $L$ implies that for each $2$-element subset $X$ of $\{a, b, c, d\}$, there is an element of $L$ in $\cl_{M / f}(X)$.

For such a set $X$, $x$ is contained in $\cl_{M / f}(X)$ only if $r_{M / x}(X \cup \{f\}) = 2$. Since there is a unique triangle in $N$, there is exactly one $2$-element subset $X \subset \{a, b, c, d\}$ with $x \in \cl_{M / f}(X)$.
Hence there are exactly five $2$-element subsets $X \subset \{a, b, c, d\}$ such that $\cl_{M / f}(X)$ contains one of $u$, $v$, or $w$.
This implies that $M / f \d e, x \cong F_7^-$.
\end{proof}

The following proposition is a corollary of \autoref{thm:threeexcludedminors} along with Lemmas~\ref{lem:nou26}, \ref{lem:nou46}, and \ref{lem:nop6} and the fact that $P_8''$ has no four-point line.

\begin{proposition} \label{prop:ternaryorquaternary}
If $M$ is a $3$-connected matroid with a modular $4$-point line and no $F_7^-$- or $(F_7^-)^*$-minor, then either $M$ is quaternary or $M$ is isomorphic to a minor of $S(5, 6, 12)$.
\end{proposition}

We now finish the proof of our main theorem, \autoref{thm:4ptlinetheorem}. 

\fourpointlinetheorem*

\begin{proof}
We let $M$ be a $3$-connected matroid with a modular $4$-point line.
If $M$ has no $F_7^-$- or $(F_7^-)^*$-minor, then by \autoref{prop:ternaryorquaternary} it is representable over $\GF(3)$ or $\GF(4)$ (since $S(5, 6, 12)$ is a ternary matroid).
Otherwise, we may assume that $M$ has no $F_7$-minor.
Then \autoref{lem:U25givesfano} implies that $M$ has no $U_{2,5}$-minor. If $r(M) = 2$ then $M$ is ternary; otherwise $M$ has rank and corank at least three so \autoref{lem:U25U35} implies that $M$ has no $U_{3,5}$-minor. Also, $M$ has no $F_7^*$-minor by the dual of \autoref{lem:F7F7*}.
It now follows from \autoref{thm:reidstheorem} that $M$ is ternary.
\end{proof}

An unresolved question is whether we can guarantee the existence of an $F_7^-$-minor in outcome (\ref*{out:4ptlinetheorem-last}) of the theorem rather than just an $F_7^-$- or an $(F_7^-)^*$-minor. This would make it symmetric with respect to $F_7$ and $F_7^-$. It would be the case if we could prove that every $3$-connected matroid $M$ with a modular $4$-point line and an $(F_7^-)^*$-minor has an $F_7^-$-minor.
The techniques used to prove \autoref{lem:nop6} might settle this question; however, it is a more difficult problem, particularly because $(F_7^-)^*$ is larger than $P_6$ and because $(F_7^-)^*$ has no $U_{2,5}$--minor so we can not make use of \autoref{lem:U25givesfano}.

\section*{Acknowledgements}

I thank Jim Geelen for helpful discussions and advice about the proof.


\begin{thebibliography}{13}

\bibitem{Bixby} Robert E. Bixby, ``On Reid's characterization of the ternary matroids'', \emph{J. Combin. Theory Ser. B} \textbf{26} (1979), 174--204.

\bibitem{Brylawski} Tom Brylawski, ``Modular constructions for combinatorial geometries'', \emph{Trans. Amer. Math. Soc.} \textbf{203} (1975), 1--44.

\bibitem{Dowling} T. A. Dowling, ``A class of geometric lattices based on finite groups'', \emph{J. Combin. Theory Ser. B} \textbf{14} (1973), 61--86.

\bibitem{GeelenGerardsWhittle} Jim Geelen, Bert Gerards, and Geoff Whittle, ``Structure in minor-closed classes of matroids'', in Surveys in Combinatorics 2013 (Simon Blackburn, Stefanie Gerke and Mark Wildon eds.) \emph{London Mathematical Society Lecture Notes Series 409} Cambridge University Press, Cambridge.

\bibitem{modularplanepaper} Jim Geelen and Rohan Kapadia, ``Representation of matroids with a modular plane'', \emph{submitted}.

\bibitem{GeelenOxleyVertiganWhittle} J. F. Geelen, J. G. Oxley, D. L. Vertigan, and G. P. Whittle, ``On the excluded minors for quaternary matroids'', \emph{J. Combin. Theory Ser. B} \textbf{80} (2000), 57--68.

\bibitem{Oxley:acharacterization} James G. Oxley, ``A characterization of certain excluded-minor classes of matroids'', \emph{European J. Combin.} \textbf{10} (1989), 275--279. 

\bibitem{Oxley} James Oxley, \emph{Matroid Theory}, Second edition, Oxford University Press, New York, 2011.

\bibitem{Seymour:decomposition} P. D. Seymour, ``Decomposition of regular matroids'', \emph{J. Combin. Theory Ser. B} \textbf{28} (1980), 305--359.

\bibitem{Seymour:ternaryexcludedminors} P. D. Seymour, ``Matroid representation over GF(3)'', \emph{J. Combin. Theory Ser. B} \textbf{26} (1979), 159--173.

\bibitem{Seymour:u24roundedness} P. D. Seymour, ``On minors of non-binary matroids'', \emph{Combinatorica} \textbf{1} (1981), 75--78.

\bibitem{Seymour:triplesinmatroidcircuits} P. D. Seymour, ``Triples in matroid circuits'', \emph{European J. Combin.} \textbf{7} (1986), 177--185.

\bibitem{Tutte} W. T. Tutte, ``Menger's theorem for matroids'', \emph{J. Res. Nat. Bur. Standards Sect. B} \textbf{69B} (1965), 49--53.

\end{thebibliography}
\end{document}